\documentclass{article}
\usepackage[latin2]{inputenc} 
\usepackage {latexsym} \usepackage {amsmath} \usepackage {amsfonts} 
\usepackage {amssymb}
\usepackage {times}
\usepackage {graphicx}

\let\phi\varphi

\def\Z{{\mathbb Z}}

\def\Pt{\mathop{\rm Pt}}

\def\twojoin{\mathord{=}}
\let\=\twojoin
\def\threejoin{\mathord{\equiv}}

\def\calG{{\cal G}}
\newcommand{\FF}[1][]{\mathrel{\xrightarrow{\ifx @#1@ FF\else FF_{#1} \fi}}}
\newcommand{\nFF}[1][]{\thickspace\not\negthickspace\xrightarrow{\ifx @#1@ FF\else FF_{#1} \fi}}

\def\cc{\mathrel{\xrightarrow{cc}}}
\def\ncc{\mathrel{\thickspace\not\negthickspace\xrightarrow{cc}}}

\def\hom{\mathrel{\xrightarrow{hom}}}

\newtheorem{theorem}{Theorem}[section]
\newtheorem{lemma}[theorem]{Lemma}
\newtheorem{corollary}[theorem]{Corollary}

\newtheorem{conjecture}[theorem]{Conjecture}

\newtheorem{question}[theorem]{Question}

\newenvironment{proof}{\par\medskip\noindent{\bf Proof: }}
  {\unskip\hfill$\Box$\par\medskip}
\newenvironment{proofof}{\par\medskip\proofof}
  {\unskip\hfill$\Box$\par\medskip}
\def\proofof #1{\noindent{\bf Proof of #1:}}

\title{Cycle-continuous mappings -- order structure} 
\author{
Robert Šámal \thanks{Supported by grant GA \v{C}R P201/10/P337.}
\\
\small
 Computer Science Institute, Charles University
\\
\small
 email: {\tt samal@iuuk.mff.cuni.cz} 
}

\begin{document}
\date{}
\maketitle

\begin{abstract}
Given two graphs, a mapping between their edge-sets is 
\emph{cycle-continuous}, if the preimage of every cycle is a cycle. 
Answering a question of DeVos, Nešetřil, and Raspaud, we prove that there 
exists an infinite set of graphs with no cycle-continuous mapping between them. 
Further extending this result, we show that every countable poset can be represented 
by graphs and existence of cycle-continuous mappings between them. 
\end{abstract}

\section{Introduction}
\label{sec:group}

Many questions at the core of graph theory can be formulated as 
questions about cycles or more generally about flows on graphs. 
Examples are Cycle Double Cover conjecture, Berge-Fulkerson conjecture, 
and Tutte's 3-Flow, 4-Flow, and 5-Flow conjectures. 
For a detailed treatment of this area the reader may refer to 
\cite{seymour-handbook} or \cite{zhang-book}. 

As an approach to these problems Jaeger \cite{Jaeger} and 
DeVos, Nešetřil, and Raspaud \cite{DNR} 
defined a notion of graph morphism continuous with respect to 
group-valued flows. In this paper we restrict ourselves to the case of 
$\Z_2$-flows, that is to cycles. Thus, the following is the 
principal notion we study in this paper: 

Given graphs (parallel edges or loops allowed) $G$ and $H$, a mapping $f: E(G) \to E(H)$ is called 
\emph{cycle-continuous}, if for every cycle $C \subseteq E(H)$, 
the preimage $f^{-1}(C)$ is a cycle in~$G$. We emphasize, that 
by a \emph{cycle} we understand (as is common in this area) 
a set of edges such that every vertex is adjacent with an even number of them. 
For shortness we sometimes call cycle-continuous mappings just $cc$ mappings. 

The fact that $f$ is a $cc$ mapping from~$G$ to~$H$ is denoted by 
$f: G \cc H$. If we just need to say that there exists a $cc$ mapping 
from~$G$ to~$H$, we write $G \cc H$; inspired by the notation common 
in graph homomorphisms. 

With the definition covered, we mention the main conjecture 
describing the properties of $cc$ mappings. 

\begin{conjecture}[Jaeger]\label{c:Ptcol} 
For every bridgeless graph $G$ we have $G\cc\Pt$, where $\Pt$ denotes the Petersen graph. 
\end{conjecture}

If true, this would imply many conjectures in the area. 
To illustrate this, suppose we want to find a 5-tuple of cycles in a graph~$G$ covering each of its edges twice 
(this is conjectured to exist by 5-Cycle double cover conjecture \cite{seymour-cdc,szekeres-cdc,celmins-thesis}).  
Further, suppose $f: G \cc \Pt$. We can use $C_1$, \dots, $C_5$ --- 
a 5-tuple of cycles in the Petersen graph double-covering its edges--- and then it is easy to check that 
$f^{-1}(C_1)$, \dots, $f^{-1}(C_5)$ have the same property in~$G$. 

DeVos et al.\ \cite{DNR} study this notion further and ask the following question 
about the structure of cycle-continuous mappings. 
We say that graphs $G$, $G'$ are $cc$-incomparable if there is no $cc$ mapping between them, 
that is $G\ncc G'$ and $G' \ncc G$. 

\begin{question}[\cite{DNR}]\label{q:antichain} 
Is there an infinite set $\calG$ of bridgeless graphs such that every two of them 
are $cc$-incomparable? 
\end{question} 

A negative answer to this would suggest a way to attack Conjecture~\ref{c:Ptcol}. 
DeVos et al. \cite{DNR} prove in their Theorem~2.9 that if there is no infinite set 
as in the above question, neither an infinite chain 
$G_1 \cc G_2 \cc G_3 \cc \cdots$ (such that $G_{n+1} \ncc G_n$ for all $n$), 
then there is a single graph $H$ such that for every other bridgeless graph~$G$ 
we have $G \cc H$. 

DeVos et al.~\cite{DNR} also show that arbitrary large sets of $cc$-incomparable graphs exist. 
Their proof is based on the notion of critical snarks and on 
Lemma~\ref{l:critical}; these will be crucial also for our proof. 

We will show, that the answer to Conjecture~\ref{q:antichain} is positive. 
Thus, the following is the first main result of this paper. 

\begin{theorem}
There is an infinite set $\calG$ of cubic bridgeless graphs such that every two of them 
are $cc$-incomparable. 
\end{theorem}

While this definitely shouldn't be interpreted as an indication that Conjecture~\ref{c:Ptcol} 
is false, it eliminates some easy paths towards the possible proof of it. 
As a further indication of the complexity of the structure of $cc$ mappings, we 
study the order that $cc$ mappings induce on graphs. 

When given a set of objects and morphisms between them, it is standard to 
consider a poset in which $x \le y$ iff there is a mapping from~$x$ to~$y$. 
In this sense we can speak about the poset of $cc$~mappings and ask what 
subposets it contains. The above theorem can be restated: this poset 
contains infinite antichains (poset with no relation). 
It is perhaps surprising, that this poset in fact contains all other countable posets. 

\begin{theorem}
Every countable (finite or infinite) poset can be represented by 
a set of graphs and existence of cycle-continuous mappings between them. 
\end{theorem}

To further illustrate the topic, we briefly mention related concept of 
cut-contin\-uous mappings. By a \emph{cut} we mean a set of edges of form~$\delta(U)$ -- all edges
leaving some set~$U$ of vertices. Such set may be empty (if $U$ is empty), but if 
it is not, it will disconnect the graph. However, not all edge-sets that disconnect the 
graph are cuts in our sense! A set~$\delta(\{v\})$ will be called \emph{elementary edge-cut} 
determined by vertex~$v$. 

A mapping $f: E(G) \to E(H)$ is \emph{cut-continuous} if the preimage of every cut is a cut. 
Cut-continuous mappings behave in many contexts similarly as homomorphisms (see 
\cite{NS-TT1, NS-TT2}), in particular Question~\ref{q:antichain} would be trivial for cut-continuous mappings. 
The cycle-continuous mappings, on the other hand, have been hard to tame so far, perhaps 
because of their connection with so many longstanding conjectures.

\section{Properties of cycle-continuous mappings} 

\subsection{Basics} 

Before we describe our construction, we introduce basic properties 
of cycle-contin\-uous mappings. Most of them are easy and may implicitly appear before, 
but we list all that we need for the reader's convenience. 
By a graph we mean a multigraph with loops and parallel edges allowed.

The following is well-known. 

\begin{lemma}
The following are equivalent properties of a graph~$G$: 
\begin{itemize} 
  \item $G \cc K_2^3$ (here $K_2^3$ is the graph with two vertices and three parallel edges). 
  \item $G$ has a $4$-NZF
  \item (if $G$ is cubic) $G$ admits a 3-edge-coloring 
\end{itemize} 
\end{lemma}

A cubic connected bridgeless graph is called a \emph{snark} if
is is not 3-edge-colorable. In view of the above lemma, this happens 
precisely when $G \ncc K_2^3$. 

The next result can be proved by using the cut-cycle duality, see \cite{DNR}. 

\begin{lemma}\label{l:altdef} 
Let $f: E(G) \to E(H)$ be a mapping. 
Mapping $f$ is cycle-continuous if and only if for every cut~$C$
in~$G$, the set of edges of~$H$, to which an odd number of edges of~$C$
maps, is a cut.

Moreover, it is sufficient to verify the condition for all cuts determined 
by a single vertex. 
\end{lemma} 

\begin{corollary} \label{c:threecut}
Suppose $f: G \cc H$ and $H$ is bridgeless. 
Then for every 3-edge-cut $\{e_1, e_2, e_3\}$ the set 
$\{f(e_1), f(e_2), f(e_3)\}$ is a 3-edge-cut. 
\end{corollary}

\begin{corollary} \label{c:localiso}
Let a mapping $f: E(G) \to E(H)$ be such that 
for each vertex~$v$ of~$G$, it maps all edges incident with~$v$ to 
all edges incident with some vertex of~$H$. 
Then $f$ is cycle-continuous. 
\end{corollary}

Corollary~\ref{c:localiso} explains a frequently mentioned version of 
Conjecture~\ref{c:Ptcol}: every cubic bridgeless graph~$G$ has a mapping 
$f: E(G) \to E(Pt)$ such that adjacent edges are mapped to adjacent edges. 

$G/e$ is the (multi)graph obtained by identifying both ends of an edge $e \in E(G)$, 
erasing the loop that results from~$e$, but keeping possible other loops and multiple edges. 
The following appears in \cite{DNR}, we include the easy proof for the reader's convenience, 
as it illustrates later proofs in our treatment.  

\begin{lemma}\label{l:contract} 
$G/e \cc G$ for every graph $G$ and $e \in E(G)$. 
\end{lemma} 

\begin{proof}
We define $f:E(G/e) \to E(G)$ in the natural way: for an edge $a$ of~$G/e$ we 
let $f(a)$ be the edge of~$G$ from which $a$ was created. 
To prove that $f$ is a $cc$ mapping, we only need to observe, that for every 
cycle~$C$ in~$G$, $C/e$ is a cycle in~$G/e$. 
\end{proof}

We shall call the mapping $cc$ mapping from~$G/e$ to~$G$ a \emph{natural inclusion}. 


\subsection{Properties of a 2-join} 

In this and the next section we will describe two common construction of 
snarks. While the constructions are known (see, e.g., \cite{zhang-book}), 
the relation to cycle-continuous mappings has not been investigated elsewhere, and 
is crucial to our result. 

The first construction can be informally described as adding a ``gadget'' 
on an edge of a graph. Formally, let $G_1$, $G_2$ be graphs, and 
let $e_i=x_iy_i$ be an edge of~$G_i$. We delete edge~$e_i$ from~$G_i$ (for $i=1, 2$), 
and connect the two graphs by adding two new edges $x_1 x_2$ and $y_1 y_2$. 
The resulting graph will be called the \emph{2-join} of the graphs $G_1$, $G_2$ 
(some authors call this a 2-cut construction); it will be denoted by $G_1\=G_2$. 
We note that the resulting graph depends on our choice of the edges~$x_iy_i$, but 
for our purposes this coarse description will suffice. 

\begin{lemma} \label{l:twoincl} 
For every graphs~$G_1$, $G_2$ we have $G_i \cc G_1\=G_2$ for $i \in \{1,2\}$. 
\end{lemma} 

\begin{proof}
We consider the natural mapping from~$E(G_i)$ to~$E(G_1 \= G_2)$: the edge $e_i$ of~$G_i$
(that is deleted in the 2-join) will be mapped to $x_1 x_2$. 
To show the mapping is cycle-continuous, 
we use Lemma~\ref{l:altdef}: cut $\delta(\{v\})$ in~$G_i$ correspond either to the same vertex cut in~$G_1 \threejoin G_2$, 
unless $v=y_1$. The cut $\delta(\{y_1\})$ is, however, mapped also to a 3-edge-cut, which finishes the proof. 
\end{proof}

\begin{lemma}\label{l:twoproj}
Let $G_1$, $G_2$ be any graphs.  
Let $K$ be an edge-transitive graph. 
Then $G_1\= G_2 \cc K$ if and only if $G_1 \cc K$ and $G_2 \cc K$. 
\end{lemma} 

\begin{proof}
For the forward implication it is enough to use Lemma~\ref{l:twoincl}. 
For the other one: consider cycle-continuous mappings $f_i: E(G_i) \to E(K)$, 
let $e_i = x_i y_i$ be the edges on which the 2-join operation is performed. 
As $K$~is edge-transitive, we may assume that $f_1(e_1) = f_2(e_2)$. Thus, we may 
define $f; E(G_1 \= G_2) \to E(K)$ in a natural way: $f(x_1x_2) = f(y_1y_2) = f_1(e_1)$ 
(which equals $f_2(e_2)$), 
and $f(e) = f_i(e)$ whenever $e \not= e_i$ is an edge of $G_i$. 
Corollary~\ref{c:localiso} implies easily that $f$~is cycle-continuous. 
\end{proof}

As an immediate corollary we get the following classical result 
about snarks and 2-joins: 

\begin{corollary} 
If $G_1$, $G_2$ are bridgeless cubic. Then $G_1 \= G_2$ is a snark whenever 
at least one of $G_1$, $G_2$ is a snark. 
\end{corollary} 

Another easy corollary of Lemma~\ref{l:twoproj} is that minimal counterexample (if it exists) 
to Conjecture~\ref{c:Ptcol} does not contain a nontrivial 2-edge-cut. 

\begin{corollary} \label{c:twotoPt} 
Let $G_1$, $G_2$ be cubic bridgeless graphs. If 
$G_1 \= G_2 \ncc \Pt$ then $G_i \ncc \Pt$\, 
for some $i\in \{1, 2\}$. 
\end{corollary} 

\subsection{Properties of a 3-join} 

A \emph{3-join} (also called 3-cut construction) is a method 
to create new snarks -- ones that contain nontrivial 3-edge cuts. 
One way to view this is that we replace a vertex in a graph by a ``gadget'' 
created from another graph. 

To be more precise, 
we consider graphs~$G_1$ and~$G_2$, delete a vertex $u_i$ of each~$G_i$, and add 
a matching between neighbors of former vertices $u_1$ and~$u_2$. 
The resulting (cubic) graph in general depends on our choice of~$u_i$'s, 
and of the matching, but in our applications it either will not matter, or will 
be discussed in advance, so we do not introduce any special notation for this. 
We use $G_1 \threejoin G_2$ to denote (any of) the resulting graph(s); we call in 
the \emph{3-join of~$G_1$ and~$G_2$}. 
Connecting edges of the 3-join are the three edges we added to connect~$G_1$ and~$G_2$. 

We collect several easy properties of the 3-join operation. 

\begin{lemma}\label{l:threeincl} 
For any graphs $G_1$, $G_2$ we have $G_i \cc G_1\threejoin G_2$ for $i=1, 2$. 
\end{lemma} 

\begin{proof}
We consider the natural mapping from~$E(G_i)$ to~$E(G_1 \threejoin G_2)$. To show it is cycle-continuous, 
we use Lemma~\ref{l:altdef}: any vertex cut in~$G_i$ correspond either to a vertex cut in~$G_1 \threejoin G_2$ 
or to the connecting edges that also form a 3-edge cut. 
\end{proof}

We shall call the cycle-continuous mapping from~$G_i$ to~$G_1 \threejoin G_2$ that is used in the above lemma 
a \emph{natural inclusion}. 

\begin{lemma}\label{l:threeproj}
Let $G_1$, $G_2$ be any graphs.  
Let $K$ be a cyclically 4-edge-connected cubic graph with the following symmetry property: 
\begin{center}
Whenever $u_i$ ($i=1,2$) is a vertex and $x_{i,1}$, $x_{i,2}$, $x_{i,3}$ 
is an ordering of $N(u_i)$, there is an automorphism $f$ of~$K$ such that 
$f(u_1) = u_2$ and $f(x_{1,j}) = x_{2,j}$ for $j=1,2,3$. 
\end{center}
Then $G_1\threejoin G_2 \cc K$ if and only if $G_1 \cc K$ and $G_2 \cc K$. 
\end{lemma} 

\begin{proof}
The `only if' part follows from Lemma~\ref{l:threeincl}. For the other direction, 
consider any $f_i:G_i \cc K$ ($i=1,2$). 
Also let $v_i$ be the vertex of~$G_i$ deleted in the 3-join operation, and let 
$a_i$, $b_i$, $c_i$ be the edges incident to~$v_i$, labeled in an order compatible 
with the matching chosen in the 3-join operation. 

Using Lemma~\ref{l:altdef}, we see that $S_i=\{f_i(a_i), f_i(b_i), f_i(c_i)\}$ 
is a 3-edge cut in~$K$. As $K$ is cyclically 4-edge-connected, $S_i$ is a cut 
around some vertex of~$K$. The symmetry property together with the fact that 
isomorphism induces a $cc$ mapping implies, that we can assume that 
$S_1 = S_2$, and even  
$f_1(a_1) = f_2(a_2)$, $f_1(b_1) = f_2(b_2)$, and $f_1(c_1) = f_2(c_2)$. 
Consequently, we may define a mapping $f: G_1\threejoin G_2 \cc K$ in a natural way: 
if $e$ is an edge of~$G_i$,  we let $f(e) = f_i(e)$. Because of the above 
assumption, the connecting edges are mapped consistently. 
To verify that $f$~is cycle-continuous, we use Corollary~\ref{c:localiso}. 
\end{proof}

As an immediate corollary we get the following classical result 
about snarks and 3-joins: 

\begin{corollary} 
Let $G_1$, $G_2$ be cubic bridgeless graphs. Then $G_1 \threejoin G_2$ 
is a snark, iff at least one of $G_1$, $G_2$ is a snark. 
\end{corollary} 

\begin{proof}
Apply Lemma~\ref{l:threeproj} for $K = K_2^3$. 
\end{proof}

As another easy application, we observe that minimal counterexample (if it exists) 
to Conjecture~\ref{c:Ptcol} does not contain a nontrivial 3-edge-cut. 

\begin{corollary} \label{c:threetoPt} 
Let $G_1$, $G_2$ be cubic bridgeless graphs. If 
$G_1 \threejoin G_2 \ncc \Pt$ then $G_i \ncc \Pt$\,   
for some $i\in \{1, 2\}$. 
\end{corollary} 

The above notwithstanding, we proceed to study the structure 
of cycle-contin\-uous mapping in graphs with 3-edge-cuts, for two reasons: first we believe, it provides 
insights that might be useful in further progress towards solving Conjecture~\ref{c:Ptcol}, 
second, we find it has an independent interest. 

\begin{lemma}\label{l:threeincomp} 
Let $G_1$, $G_2$ be $cc$-incomparable snarks. Then \\ 
$G_1\threejoin G_2 \ncc G_i$\, for each $i \in \{1, 2\}$. 
\end{lemma} 

\begin{proof}
Immediate from Lemma~\ref{l:threeincl}. 
\end{proof}

\section{The proof} 

\subsection{Critical snarks} 

For our construction we will need the following notion of criticality of snarks. 
It appears in DeVos et al.\cite{DNR}, see also \cite{daSilva}, where these graphs 
are called flow-critical snarks. 

Recall a graph $G$ is a snark if $G \ncc K_2^3$, where $K_2^3$ is a graph formed 
by two vertices and three parallel edges. 
We say $G$ is a \emph{critical snark} if for every edge~$e$ of~$G$
we have $G-e \cc K_2^3$. 
(Equivalently \cite{daSilva}, $G/e \cc K_2^3$.)  

The following lemma is a basis of our control over cycle-continuous mappings between graphs 
in our construction. 

\begin{lemma}[\cite{DNR}] \label{l:critical} 
Let $G$, $H$ be cyclically 4-edge-connected cubic graphs, 
both of which are critical snarks, suppose that $|E(G)| = |E(H)|$. 
Then $G \cc H$ iff $G \cong H$. Moreover, every cycle-continuous mapping is a bijection 
that is induced by an isomorphism of~$G$ and~$H$. 
\end{lemma} 

DeVos et al. \cite{DNR} claim, that if $G$ is critical then the dot product of~$G$ 
and the Petersen graph is critical as well (see \cite{zhang-book} for the definition 
of dot product). This allows (by different ways of taking the dot product) to 
create arbitrary large set of nonisomorphic critical snarks with the same number 
of vertices. However, this claim is not proved there, thus we will only 
use the following weaker fact. 

\begin{lemma}\label{l:Blanusasnarks} 
There are two snarks $B_1$, $B_2$ with 18~vertices, that are critical 
and nonisomorphic. Moreover, none of $B_1$, $B_2$ is vertex transitive; 
in particular, there is no isomorphism $f: V(B_2) \to V(B_2)$ for which 
$f(a) = b$. 
\end{lemma} 

\begin{proof} 
It is well-known that there are two nonisomorphic snarks on 18 vertices, called Blanuša 
snarks (Figure~\ref{fig:Blanusa}), let us use $B_i$ ($i=1,2$) to denote them. 
To prove criticality, we shall use the well-known fact \cite{hog}, that there is 
no triangle-free snark on 16 vertices. 

For any edge $e \in B_i$, $B_i-e$ is a subdivision of a cubic graph~$G$ on 16 vertices. 
As the girth of~$B_i$ is~$5$, the girth of~$G$ is at least~$4$, so $G$ is not a snark, 
and $G \cc K_2^3$. As $B_i-e$ is a subdivision of~$G$, we have also 
$B_i-e \cc K_2^3$, so $B_i$ is indeed critical. 

It is well-known that $B_i$'s are not vertex-transitive. For an easy proof for $B_2$ 
(this is all we will use) observe that in Figure~\ref{fig:Blanusa}, 
vertex~$b$ is adjacent with an edge contained in no 5-cycle, while vertex~$a$ is not. 
\end{proof} 

\begin{figure}[ht] 
  \centerline{\includegraphics[scale=0.80,page=3]{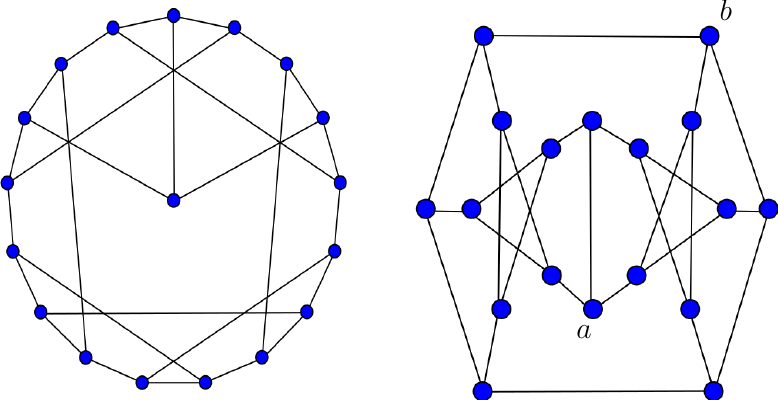}} 
  \caption {Blanuša snarks (see Lemma~\ref{l:Blanusasnarks}). The graph $B_2$ in the right 
    is not vertex transitive, in particular vertex~$a$ cannot be mapped to~$b$ by an isomorphism.
    Image by Koko90 via Wikimedia Commons, published under CC-BY-SA licence.} 
  \label{fig:Blanusa}
\end{figure}
 
\subsection{Tree of snarks} 
\label{s:snarktree} 

Let $\calG = \{ G_1, \dots, G_n\}$ be a family of critical snarks of the same size, 
so that for $i \ne j$ graphs $G_i$ and $G_j$ are not isomorphic (equivalently: 
$G_i \ncc G_j$ and $G_j \ncc G_i$). 

Let $T$ be a tree with a vertex coloring (not necessarily proper) $c:V(T) \to [n]$. 
We denote by $T(\calG)$ a family of graphs that can be obtained by replacing 
each $v \in V(T)$ by a copy of~$G_{c(v)}$ and performing a 3-join on each edge; see Fig.~\ref{fig:treesnarks} for an illustration. 
There are in general many graphs that can be constructed in this way, depending on 
which vertices one chooses for the 3-join operations. 

\begin{figure}[ht] 
  \includegraphics[page=2]{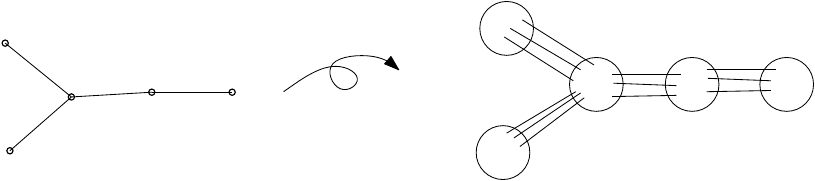}
  \caption {Illustration of the ``tree-snark'' construction.} 
  \label{fig:treesnarks}
\end{figure}

More precisely, for each $v \in V(G)$ we fix a bijection $r_v$ from~$N_T(v)$ to an 
\emph{independent set} $A_v$ in $G_{c(v)}$, we also specify an ordering of edges going 
out of vertices of~$A_v$. 
Next, we split each vertex $w$ in~$A_v$ into three degree~1 vertices; these 
will be denoted by~$w_1$, $w_2$, $w_3$. 
For each edge~$uv$ of~$T$ we identify vertices 
$r_u(v)_i$ with $r_v(u)_i$ for $i=1, 2, 3$. 
Finally, we suppress all vertices of degree~2. 

If $H$ is a graph in $T(\calG)$ and $v$ a vertex of~$T$, 
we let $H_v$~denote a ``copy'' of $G_{c(v)}$: 
subgraph of $H$ consisting of a copy of 
$G_{c(v)} - A_v$ together with the incident edges and 
neighboring vertices in~$H$. 
Further, we let $\iota_v$ denote the natural inclusion 
of~$G_{c(v)}$ into~$H$, which maps bijectively on~$H_v$. 

We define $\bar H_v$ to be the graph $H$ with all edges outside of~$H_v$ contracted. 
In other words, $\bar H_v$ is truly an isomorphic copy of~$G_{c(v)}$. 
Further, $H_{u,v}$ will denote the three edges in the intersection $H_u \cap H_v$. 

The following lemma and theorem are the key to our construction. 

\begin{lemma} \label{l:snarktree}
Let $\calG$ be as above. 
Take $H \in T(\calG)$ and $K \in \calG$. 
Then $K \cc H$ iff $K \cong G_i$\, for some $G_i \in \calG$ such that color $i$ is used on~$T$. 
Moreover, all mappings $K \cc H$ are an isomorphism on~$K$ composed with $\iota_v$ 
for some $v \in V(G)$ for which $c(v)=i$. 
\end{lemma} 

\begin{proof} 
Consider a cycle-continuous mapping $f: E(K) \to E(H)$, let $R$ be the set of edges in the range of~$f$. 
Suppose first, that $R$ is exactly the edge set of one of the graphs $H_v$. 
Then $f: E(K) \to E(H_v)$ is also $cc$, the rest follows by Lemma~\ref{l:critical}. Suppose next, that 
for every $v$, some edge of $H_v$ is not in~$R$; let $H'_v$ be the subgraph of~$\bar H_v$ induced by~$R$. 
As each graph of~$\calG$ is critical, each graph~$H'_v$ has $cc$ mapping to~$K_2^3$. 
The graph $H'$ (subgraph of~$H$ induced by~$R$) is produced from the graphs $H'_v$ by 
2-join and 3-join operations, which implies that $K \cc H' \cc K_2^3$. This is a contradiction, 
as $K$~is a snark. 
\end{proof} 

\begin{theorem} \label{t:snarktrees}  
Let $T_1$, $T_2$ be two trees and let $c_i:V(T_i) \to [n]$ be arbitrary colorings. 
Let $\calG$ be as above. 

Suppose $H_i \in T_i(\calG)$ for $i=1, 2$. 
Every $cc$~mapping $g: H_1 \cc H_2$ is guided by a homomorphism $f: T_1 \to T_2$ of reflexive colored graphs: 
There is a mapping $f : V(T_1) \to V(T_2)$ such that 
\begin{itemize} 
\item $c_2(f(v)) = c_1(v)$ ($f$ respects colors), and 
\item if $uv$ is an edge of~$T_1$, then $f(u)f(v)$ is an edge of $T_2$ or $f(u) = f(v)$. 
In the first case, $g$ maps $H_{u,v}$ to $H_{f(u), f(v)}$. 
In the second one, $H_{u,v}$ is mapped to some $H_{f(u), v'}$. 
\end{itemize} 

Moreover, $g$ induces a mapping $(\bar H_1)_v$ to $(\bar H_2)_{f(v)}$ that is cycle-continuous. 
\end{theorem} 

\begin{proof} 
For a vertex~$v$ of~$T_1$, consider the composition of~$g$ with~$\iota_v$. It is a cycle-continuous mapping 
from $G_{c_1(v)}$ to~$H_2$. By Lemma~\ref{l:snarktree} this mapping is onto some~$(H_2)_{v'}$, for which 
$c_2(v') = c_1(v)$. We put $f(v) = v'$. Next, for an edge~$uv$ of~$T_1$ we observe that 
$(H_1)_{u,v}$ is a part of both~$(H_1)_u$ and~$(H_1)_v$, thus $(H_2)_{f(u)}$ and~$(H_2)_{f(v)}$ must have 
common edges. If follows that either $f(u)=f(v)$ or $f(u)f(v)$ is an edge of~$T_2$. The rest follows easily. 
\end{proof} 

As a corollary we obtain our first result, that already answers Question~\ref{q:antichain}. 

\begin{corollary} \label{c:infantichain} 
There is an infinite set of $cc$-incomparable graphs. 
\end{corollary}

\begin{proof} 
Let $T_n$ be a path with vertices~$\{0, 1, \dots, n\}$ colored as $1(2)^{n-1} 1$. 
We let $\calG = \{B_1, B_2\}$, where as in Lemma~\ref{l:Blanusasnarks}, $B_i$'s~denote 
the Blanuša snarks. For all vertices $v \in V(T_n)$ of degree~2 we create~$r_v$ 
so, that $r_v(v-1) = a$ and $r_v(v+1) = b$. We do not specify $A_0$ nor $A_n$, 
neither the order of edges adjacent to $a$ or~$b$. We let $H_n$ denote any of $T_n(\calG)$. 

Consider $H_m$, $H_n$, suppose that $g: H_m \cc H_n$ is $cc$~mapping. 
Let $f: V(T_m) \to V(T_n)$ be the mapping guaranteed by Theorem~\ref{t:snarktrees}. 
As $f$~respects colors, we have $\{f(0), f(m)\} = \{0, n\}$. 
Next, consider $G_i = (H_m)_i$, and $G'_j = (H_n)_j$. By Theorem~\ref{t:snarktrees} again, 
$g$~is $cc$ mapping $G_i \cc G'_{f(i)}$. 
As $G_i$ and~$G'_{f(i)}$ are isomorphic to~$B_2$, Lemma~\ref{l:Blanusasnarks} implies 
that $f(i+1) = f(i)+1$. It follows that $m=n$, which finishes the proof. 
\end{proof}

\subsection{Representing posets by cycle-continuous mappings} 

Question~\ref{q:antichain} should be understood as a question 
about how complicated is the structure of $cc$~mappings. 
Next, we provide even further indication, that the structure is 
complicated indeed. 


\begin{corollary} \label{c:poset}
Every countable (finite or infinite) poset can be represented by 
a set of graphs and existence of $cc$ mappings between them. 
\end{corollary}

\begin{figure}
  \centerline{\includegraphics[scale=1.00,page=1]{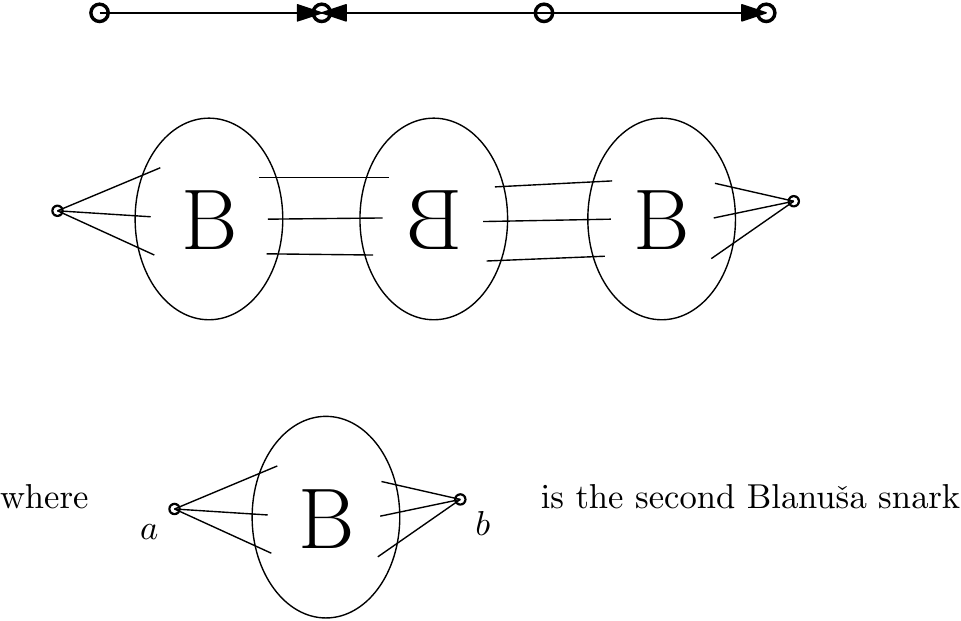}}
  \caption {Construction used for representation of arbitrary posets by $cc$ mappings.} 
  \label{fig:paths}
\end{figure}

\begin{proof} 
We use the result of Hubička and Nešetřil~\cite{HubiNes}, claiming that arbitrary 
countable posets can be represented by finite directed paths and existence 
of homomorphisms between them. 

Thus, we only need to find a mapping $m$ that to directed paths assigns cubic bridgeless graphs, 
so that $P_1 \hom P_2$ iff $m(P_1) \cc m(P_2)$. 
To do this, we use the construction depicted in Fig.~\ref{fig:paths}. Informally, we replace each directed edge 
by a copy of~$B_2$ ``from~$a$ to~$b$'' and perform a 3-join operation in-between each pair of 
adjacent edges. 
Formally, let $P$ be a path with edges (from one end to the other) $e_1$, \dots, $e_m$. 
We let $t(i)$ be the index of the edge at the tail of~$e_i$ -- that is $t(i)$ is 
either~$i-1$ (if $e_i$ goes forward with respect to our labeling) or~$i+1$. 
Note that $t(i)$ may be undefined for $i\in \{1, m\}$. 
Similarly, we define $h(i)$ to be the index of the edge adjacent to~$e_i$ at its head. 
We will use the construction from Section~\ref{s:snarktree}. Our tree~$T$ 
will be a path with vertices $1$, \dots, $m$ all colored by~$1$, our 
set of snarks will consist just of the second Blanuša snark, $\calG = \{ G_1 = B_2 \}$. 
We define $r_{i}(t(i)) = a$, and $r_{i}(h(i)) = b$, whenever $t(i)$ ($h(i)$, resp.) are defined. 
We choose an ordering of edges going out of~$a$, and~$b$; we keep this fixed for all vertices of all paths. 
Then we let $m(P)$ be the graph in~$T(\calG)$ determined by the above described choices. 

With the construction in place, we need to show that for any directed paths~$P$ and~$P'$, 
we have $P \hom P'$ if and only if $m(P) \cc m(P')$. 
The proof of the `only if' part will be direct consequence of our construction, 
the `if' part will follow from Lemma~\ref{l:snarktree}. 
To be specific, for the forward implication consider a homomorphism $f:P \hom P'$. 
Consider an edge~$e_i$ of~$P_1$, let $f(e_i) = e'_j$ (we extend the homomorphism~$f$ 
to act on edges in the natural way). As all edges were replaced by a copy of the same graph, 
we may consider an isomorphism from~$H_i$ to~$H_j$ and define~$m(f)$ to be the induced 
mapping on edges. We only need to check, that the edges in~$C := H_i \cap H_{i+1}$ 
are mapped consistently, as we are defining their image twice. 
Suppose first that $f(e_i) = f(e_{i+1}) = e'_j$ and (without loss of generality) $e_i$,~$e_{i+1}$ are meeting 
at their heads, i.e., $h(i)=i+1$ and $h(i+1)=i$. Then edges of~$C$ are mapped both times to 
the edges of $H'_j \cap H'_{h(j)}$ (and the order is the same, by our construction). 
It remains to check the case when $f(e_i)$ and~$f(e_{i+1})$ are adjacent edges, 
suppose again that $e_i$,~$e_{i+1}$ meet at their heads (other cases are analogous). 
Then the edges of~$C$ are mapped both times to~$H_{j} \cap H_{h(j)}$. 
This implies that the mappings $m(f) : E(m(P)) \to E(m(P'))$ is consistently defined 
and it maps elementary edge-cuts to elementary edge-cuts. 
Consequently, by Corollary~\ref{c:localiso}, 
$m(f)$ is cycle-continuous, which finishes the first part of the proof. 

To prove the backward implication, consider a $cc$~mapping $g: m(P) \cc m(P')$. 
Due to our construction of~$m(P)$, $m(P')$, we may use Theorem~\ref{t:snarktrees} 
to get a mapping $h:V(T) \to V(T')$, that is a homomorphism of reflexive graphs. 
Now, we may consider $h$ also as a mapping $E(P) \to E(P')$. It maps adjacent 
edges $e_{i}$, $e_{i+1}$ either to adjacent edges or to the same edge. 
It remains to check, that this mapping on edges is induced by a homomorphism 
of directed path $P \hom P'$. 
This is done in the same way, as in the proof of Corollary~\ref{c:infantichain}. 
\end{proof}

\section{Concluding remarks} 

While being a resolution to Question~\ref{q:antichain}, none of the family of examples  
we gave does violate Conjecture~\ref{c:Ptcol}: 

\begin{theorem} \label{t:toPt}
If $H \in T(\calG)$ and for every $G \in \calG$ we have $G \cc \Pt$ 
then $H \cc \Pt$. 
\end{theorem}

\begin{proof} 
It suffices to repeatedly use Corollary~\ref{c:threetoPt}. 
\end{proof} 

Still, the presented results illustrate the complexity of $cc$ mappings. 
To better understand their structure, we suggest the following questions: 

\begin{question} 
Does the poset of cubic cyclically 4-edge-connected graphs and $cc$ mappings between them 
have infinite antichains? Does it contain every countable poset as a subposet? 
How about cyclically 5-edge-connected graphs? 
\end{question} 

For the next question, recall that in a poset $(X,\le)$ an interval 
$(a,b)$ is the set $\{ x \in X: a < x < b\}$ (we must have $a<b$ for this 
definition to make sense, otherwise we call $(a,b)$ degenerated interval). 

\begin{question} 
In the poset of graphs and $cc$ mappings between them, 
is every non-degenerated interval nonempty? 
Does every non-degenerated interval contain infinite antichain? 
Does every non-degenerated interval contain every countable poset? 
\end{question} 

Note, that if Conjecture~\ref{c:Ptcol} is true, then $(Pt, K_2)$ is 
an empty but non-degen\-er\-ated interval. Is there some other? 

We also briefly note the more general definition of flow-continuous mappings, 
that extends the notion of cycle-continuous mappings: a mapping $f:E(G) \to E(H)$ is called 
$M$-flow-continuous (for an abelian group~$M$) if for every $M$-flow~$\varphi$ on~$H$, 
the composition $\phi\circ f$ is an~$M$-flow on~$G$. For detailed discussion, 
see \cite{DNR} or \cite{rsthesis}. We only mention here, that cycle-continuous 
mappings are exactly $\Z_2$-flow-continuous ones, and that 
Corollaries~\ref{c:infantichain} and~\ref{c:poset} extend trivially to $\Z$-flow-continuous mappings. 

\bibliographystyle{rs-amsplain}
{
\bibliography{FF-antichain} 
}

\end{document}